\RequirePackage{fix-cm}
\documentclass[envcountsame,envcountsect]{svjour3}
\usepackage{mathptmx,amsmath,amssymb,graphicx}
\usepackage[colorlinks,linkcolor=red,anchorcolor=blue,citecolor=blue]{hyperref}
\usepackage{epstopdf}
\smartqed

\spnewtheorem{notation}{Notation}[section]{\bf}{\it}

\begin{document}

\title{Swapping algebra, Virasoro algebra and discrete integrable system
\thanks{The research leading to these results has received funding from the European Research Council under the {\em European Community}'s seventh Framework Programme
(FP7/2007-2013)/ERC {\em  grant agreement}.}
}
\dedication{To William Goldman on the occasion of his sixtieth birthday}
\author{Zhe Sun}
\institute{Z. Sun \at
              Yau Mathematical Sciences Center, Tsinghua University, Beijing, 100084, China \\
              \email{sunzhe1985@gmail.com}           }

\maketitle

\begin{abstract}
We induce a Poisson algebra $\{\cdot,\cdot\}_{\mathcal{C}_{n,N}}$ on the configuration space $\mathcal{C}_{n,N}$ of $N$ twisted polygons in $\mathbb{RP}^{n-1}$ from the swapping algebra \cite{L12}, which is found coincide with Faddeev-Takhtajan-Volkov algebra for $n=2$. There is another Poisson algebra $\{\cdot,\cdot\}_{S2}$ on $\mathcal{C}_{2,N}$ induced from the first Adler-Gelfand-Dickey Poissson algebra by Miura transformation. By observing that these two Poisson algebras are asymptotically related to the dual to the Virasoro algebra, finally, we prove that $\{\cdot,\cdot\}_{\mathcal{C}_{2,N}}$ and  $\{\cdot,\cdot\}_{S2}$ are Schouten commute.



\end{abstract}

\section{Introduction}
The space $\mathcal{L}_{n}$ of ordinary differential operators of the form $\partial_{n} + u_2 \partial_{n-2} + ... + u_{n}$ has two Poisson structures, called the first(second resp.) Adler-Gelfand-Dickey Poissson structure \cite{A79} \cite{GD87}. These two Poisson structures are realized by Drinfeld-Sokolov reductions \cite{DS85} via the hamiltonian reductions of the dual space to the affine Kac-Moody algebra $\widehat{\operatorname{gl}(n,\mathbb{R})}$. As a consequence, these two Poisson structures are compatible and provide a bihamiltonian system. We view the configuration space $\mathcal{C}_{n,N}$ of $N$-twisted polygons in $\mathbb{RP}^{n-1}$ as the discrete version of the space $\mathcal{L}_{n}$. For $n=3$, R. Schwartz, V. Ovsienko and S. Tabachnikov \cite{SOT10} introduced the Poisson structure $\{\cdot,\cdot\}_{S3}$ corresponding to the first Adler-Gelfand-Dickey Poissson structure defined on the coordinate function ring of $\mathcal{C}_{3,N}$. They show that the pentagram map relative to this Poisson structure is completely integrable. They conjecture that there exists a Poisson structure  $\{\cdot,\cdot\}_{C3}$ on the coordinate function ring of $\mathcal{C}_{3,N}$ corresponding to the second Adler-Gelfand-Dickey Poisson structure(it is also mentioned by B. Khsein, F. Soloviev in \cite{KS13} for $n$ in general and many others).
Moreover, they conjecture that $\{\cdot,\cdot\}_{C3}$ is compatible with their Poisson structure. 
It is very interesting to consider these two conjectures for $n$ in general. But the Poisson structure corresponding to the first Adler-Gelfand-Dickey Poissson bracket on $\mathcal{C}_{n,N}$ for $n>3$ is still missing. The latter conjecture is not clear even for $n=2$.

The swapping algebra \cite{L12}, introduced by F. Labourie, is used to characterize the Atiyah-Bott-Goldman Poisson structure \cite{AB83}\cite{G84} for the Hitchin component \cite{H92} and the second Adler-Gelfand-Dickey Poisson structure for the space of real opers. Taken linearity relations into consideration, the quotient algebra---rank $n$ swapping algebra \cite{Su16} is used to characterize the Fock--Goncharov Poisson structure \cite{FG06} on cluster $\mathcal{X}_{\operatorname{PSL}(n,\mathbb{R}),\hat{S}}$ moduli space in \cite{Su15}.
In this paper, for $n\geq 2$, we define the Poisson structures on the coordinate function ring of $\mathcal{C}_{n,N}$ induced from the swapping algebra \cite{L12}, the induced Poisson structures are corresponding to the second Adler-Gelfand-Dickey Poisson structures. For $n=2$, the Poisson algebra induced from the swapping algebra is the same as Takhtajan-Faddeev-Volkov algebra, which is studied by L. A. Takhtajan, L. D. Faddeev, A. Yu. Volkov in 90s \cite{FT86} \cite{FV93} \cite{V88} \cite{V92} and many others in lattice Virasoro algebra. 

We consider a Poisson bracket $\{\cdot,\cdot\}_{S2}$ for $\mathcal{C}_{2,N}$, which is a ``degenerate'' version of the R. Schwartz et al.'s Poisson bracket $\{\cdot,\cdot\}_{S3}$ for $\mathcal{C}_{3,N}$, induced from the discrete version of the first Adler-Gelfand-Dickey Poissson bracket by Miura transformation.
By studying the asymptotic behavior of these two Poisson structures, we have
\begin{proposition}{\sc[Proposition \ref{propswapvir} \ref{propschvir} ]}
We relate asymptotically the dual of the Poisson structures $\{\cdot,\cdot\}_{\mathcal{C}_{2,N}}$ and $\{\cdot,\cdot\}_{S2}$ to the Virasoro algebra.
\end{proposition}
With these evidences, we compare these two Poisson structures, we have our main theorem.
\begin{theorem}
The Poisson structures $\{\cdot,\cdot\}_{\mathcal{C}_{2,N}}$ and $\{\cdot,\cdot\}_{S2}$ for $\mathcal{C}_{2,N}$ are compatible.
\end{theorem}
We calculate some examples for $n=3$, the two Poisson structures $\{\cdot,\cdot\}_{\mathcal{C}_{3,N}}$ and $\{\cdot,\cdot\}_{S3}$ are not compatible. 

In our further study, we will modify the definition of the induced Poisson bracket $\{\cdot,\cdot\}_{\mathcal{C}_{3,N}}$ to make it work. We hope that this paper helps to understand the conjectures in general cases.
\section{Poisson algebra on the configuration space of N-twisted polygons}
In this section, we induce a Poisson algebra on the configuration space of $N$-twisted polygons in $\mathbb{RP}^{n-1}$ from the swapping algebra. We give explicit formulas for $n=2,3$. For $n=2$, the induced Poisson algebra coincides with Takhtajan-Faddeev-Volkov algebra. For $n=3$, the induced Poisson algebra coincides with a lattice $W_3$ algebra.

\subsection{Swapping algebra revisited}
\label{sa}
In this subsection, we recall briefly some definitions about the swapping algebra introduced by F. Labourie. Our definitions here are based on Section 2 of \cite{L12}.
\begin{definition}{\sc[linking number]}
Let $(r, x, s, y)$ be a quadruple of four points in $S^1$. The {\em linking number} between $rx$ and $sy$ is
\begin{equation}
\mathcal{J}(rx, sy) = \frac{1}{2} \cdot \left( \sigma(r-x) \cdot \sigma(r-y) \cdot \sigma(y-x)
- \sigma(r-x) \cdot \sigma(r-s ) \cdot \sigma(s-x)\right),
\end{equation}
such that for any $a \in \mathbb{R}$, we define $\sigma(a)$ as follows. Remove any point $o$ different from $r,x,s,y$ in $S^1$ in order to get an interval $]0,1[$. Then the points $r,x,s,y \in S^1$ correspond to the real numbers in $]0,1[$, $\sigma(a)= -1; 0; 1$ whenever $a < 0$; $a = 0$; $a > 0$ respectively.
\end{definition}
\begin{figure}\centering
\includegraphics[width=4in]{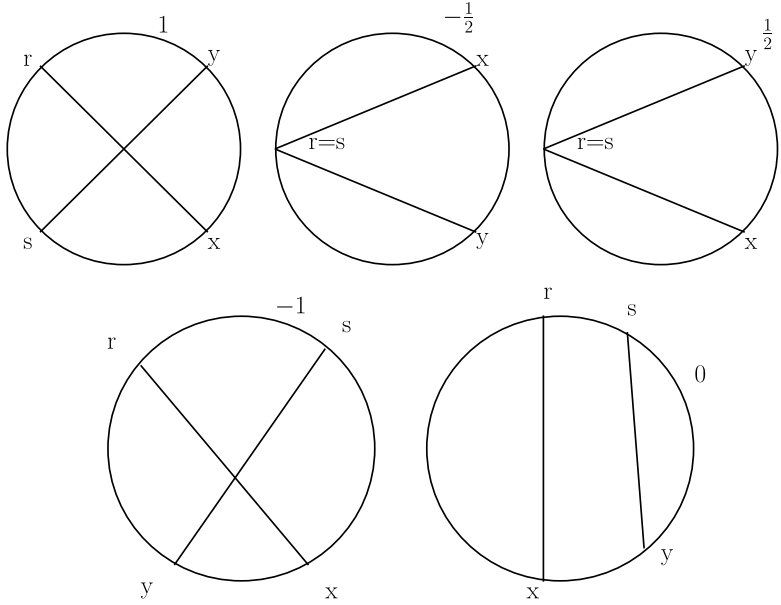}
\caption{Linking number between $rx$ and $sy$.}
\label{swapswap1}
\end{figure}
In fact, the value of $\mathcal{J}(rx, sy)$ belongs to $\{0, \pm1, \pm \frac{1}{2}\}$, depends on the corresponding positions of $r,x,s,y$, and does not depend on the choice of the point $o$. In Figure \ref{swapswap1}, we describe five possible values of $\mathcal{J}(rx, sy)$.

Let $\mathcal{P}$ be a finite subset of the circle $S^1$ provided with cyclic order. $\mathbb{K}$ is a characteristic zero field. We represent an ordered pair $(r, x)$ of $\mathcal{P}$ by the expression $rx$.

 \begin{definition}{\sc[swapping ring of $\mathcal{P}$]}
 {\em The swapping ring of $\mathcal{P}$} is the ring $$\mathcal{Z}(\mathcal{P}) := \mathbb{K}[\{xy\}_{\forall x,y \in \mathcal{P}}]/\{xx| \forall x \in \mathcal{P}\}$$ over $\mathbb{K}$, where $\{xy\}_{\forall x,y \in \mathcal{P}}$ are variables with values in $\mathbb{K}$.
 \end{definition}
 Notably, $rx = 0$ if $r = x$ in $\mathcal{Z}(\mathcal{P})$. Then we equip $\mathcal{Z}(\mathcal{P})$ with a Poisson bracket defined by F. Labourie in Section 2 of \cite{L12}.

\begin{definition}{\sc[swapping bracket]}
\label{defn2.2}
{\em The swapping bracket} over $\mathcal{Z}(\mathcal{P})$ is defined by extending the following formula on generators to $\mathcal{Z}(\mathcal{P})$ by {\em Leibniz's rule}:
\begin{equation}
\{rx, sy\} = \mathcal{J}(rx, sy) \cdot ry \cdot sx.
\end{equation}
(Here is the case for $\alpha = 0$ in Section 2 of \cite{L12}.)
\end{definition}
\begin{theorem}
\label{swapPoisson}
{\sc[F. Labourie \cite{L12}]} The swapping bracket as above verifies the Jacobi identity. So the swapping bracket defines a Poisson structure on $\mathcal{Z}(\mathcal{P})$.
\end{theorem}

\begin{definition}{\sc[swapping algebra of $\mathcal{P}$]}
{\em The swapping algebra of $\mathcal{P}$} is $\mathcal{Z}(\mathcal{P})$ equipped with the swapping bracket.
\end{definition}

\begin{definition}{\sc[swapping fraction algebra of $\mathcal{P}$]}
The {\em swapping fraction algebra of $\mathcal{P}$} is fraction ring $\mathcal{Q}(\mathcal{P})$ of $\mathcal{Z}(\mathcal{P})$ equipped with the swapping bracket.
\end{definition}
\begin{definition}{\sc[Cross fraction]}
Let $x,y,z,t$ belong to $\mathcal{P}$ so that $x \neq t$ and $y \neq z$. The {\em cross fraction} determined by $(x,y,z,t)$ is the element of $\mathcal{Q}(\mathcal{P})$:
\begin{equation}
[x,y,z,t] :=\frac{xz}{xt} \cdot \frac{yt}{yz}.
\end{equation}

Let $\mathcal{CR}(\mathcal{P}) = \{[x,y,z,t] \in \mathcal{Q}(\mathcal{P}) \; |\; \forall x,y,z,t \in \mathcal{P}, \; x \neq t, y \neq z\}$ be the set of all the cross-fractions in $\mathcal{Q}(\mathcal{P})$.
\end{definition}

\begin{definition}{\sc[swapping multifraction algebra of $\mathcal{P}$]}
Let $\mathcal{B}(\mathcal{P})$ be the subring of $\mathcal{Q}(\mathcal{P})$ generated by $\mathcal{CR}(\mathcal{P})$.
{\em The swapping multifraction algebra of $\mathcal{P}$} is $\mathcal{B}(\mathcal{P})$ equipped with the swapping bracket.
\end{definition}

\subsection{Induced Poisson structures on the configuration spaces}
\begin{definition}{\sc[Configuration space $\mathcal{C}_n$]}
Let $f$ be a map from $\mathbb{Z}$ to $\mathbb{RP}^{n-1}$ such that for any $i_1<...<i_n$ in $
\mathbb{Z}$, the images $f(i_1),...,f(i_n)$ are in general position in $\mathbb{RP}^{n-1}$. The configuration space $\mathcal{C}_n$ is the space of all these maps up to projective transformations.
\end{definition}

\begin{definition}{\sc[Configuration space of $N$-twisted polygons in $\mathbb{RP}^{n-1}$]}
For $N>n$, a {\em $N$-twisted polygon in $\mathbb{RP}^{n-1}$} is a map $f$ from $\mathbb{Z}$ to $\mathbb{RP}^{n-1}$ such that for any $k \in \mathbb{Z}$, we have $f(k+N) = M_f \cdot f(k)$ where the {\em monodromy} $M_f$ belongs to $\operatorname{PSL}(n,R)$. We say that $f$ is {\em in general position} if for any $1\leq i_1<...<i_n\leq N$, the images $f(i_1),...,f(i_n)$ in $\mathbb{RP}^{n-1}$ are in general position.
 The {\em configuration space of $N$-twisted polygons in $\mathbb{RP}^{n-1}$}, denoted by $\mathcal{C}_{n,N}$, is the space of the $N$-twisted polygons in general position in $\mathbb{RP}^{n-1}$ up to projective transformations.
\end{definition}
\begin{definition}
{\sc[Induced Poisson structure for $\mathcal{C}_{n}$]}
Without loss of ambiguity, let $\mathcal{P}=\{...,r_1,r_2,...\}$, the swapping multifraction algebra of $\mathcal{P}$ is $(\mathcal{B}(\mathcal{P}),\{,\})$. For each configuration $f$ in general position, we associate $f$ to a map $f_{n-1}$, where $f_{n-1}(r_k) = f(r_k) \wedge ...\wedge f(r_{k+n-2}) \in \mathbb{RP}^{n-1*}$ for any $k \in \mathbb{Z}$. We lift $f$ ($f_{n-1}$ resp.) to the map $\widetilde{f}$ ($\widetilde{f_{n-1}}$ resp.) with images in $\mathbb{R}^n$ ($\mathbb{R}^{n*}$ resp.). Let $\Omega$ be a volume form of $\mathbb{R}^n$. Let $\mathcal{B}'(\mathcal{P})$ be a subfraction ring of $\mathcal{Q}(\mathcal{P})$ generated by $\mathcal{CR}'(\mathcal{P})$, where $\mathcal{CR}'(\mathcal{P})$ is the set of the cross fractions such that for each factor $ab$ of the cross fraction we have $\widetilde{f}(a) \wedge \widetilde{f_{n-1}}(b)\neq 0$. We define a homomorphism $\theta$ from $\mathcal{B}'(\mathcal{P})$ to $C^\infty(\mathcal{C}_{n})$ by extending the following formula on generators:
$$\theta\left(\frac{ac}{ad} \cdot \frac{bd}{bc}\right)=\frac{\Omega\left(\widetilde{f}(a)\wedge \widetilde{f_{n-1}}(c)\right)}{\Omega\left(\widetilde{f}(a)\wedge \widetilde{f_{n-1}}(d)\right)} \cdot \frac{\Omega\left(\widetilde{f}(b)\wedge \widetilde{f_{n-1}}(d)\right)}{\Omega\left(\widetilde{f}(b)\wedge \widetilde{f_{n-1}}(c)\right)},$$
which does not depend on the lifts and the volume form $\Omega$.

Then the {\em induced Poisson bracket $\left\{,\right\}_{\mathcal{C}_n}$ on $\theta(\mathcal{B}'(\mathcal{P}))$} is 
$$\left\{\theta(e),\theta(f)\right\}_{\mathcal{C}_n} = \theta\left(\{e,f\}\right),$$
for any $e,f \in \mathcal{B}'(\mathcal{P})$.
\end{definition}
\begin{remark}
\begin{enumerate}
\item
We can modify $f_{n-1}$ such that $f_{n-1}(r_k) = f(r_{k}) \wedge  f(r_{k+j_1}) \wedge...\wedge f(r_{k+j_{n-2}}) \in \mathbb{RP}^{n-1*}$, where $j_1<...<j_{n-2}$ in $\mathbb{Z}$, to get different induced Poisson brackets.
\item The kernel of $\theta \circ \psi$ is characterized by the $(n+1) \times (n+1)$ determinant relations in \cite{Su16}. 
\end{enumerate}
\end{remark}
\begin{definition}
{\sc[Induced Poisson structure for $\mathcal{C}_{n,N}$]}
The natural embedding $\psi$ from $\mathcal{C}_{n,N}$ to $\mathcal{C}_{n}$ induces a map $\psi$ from $C^\infty(\mathcal{C}_{n})$ to $C^\infty(\mathcal{C}_{n,N})$,  the {\em induced Poisson bracket $\left\{,\right\}_{\mathcal{C}_{n,N}}$ on $\psi(\theta(\mathcal{B}'(\mathcal{P})))$}  is 
$$\left\{\psi(g),\psi(h)\right\}_{\mathcal{C}_{n,N}} = \psi\left(\{g,h\}_{\mathcal{C}_n}\right),$$
for any $g,h \in \theta(\mathcal{B}'(\mathcal{P}))$.
\end{definition}
\subsubsection{Case n=2}
\begin{notation}
Let
\begin{equation}
[a, b, c, d]:=\frac{a-c}{a-d}\cdot \frac{b-d}{b-c},
\end{equation}
\begin{equation}
[x,y]:=\frac{\{x,y\}}{xy}.
\end{equation}
\end{notation}

\begin{definition}{\sc[coordinate system of $\mathcal{C}_{2,N}$]}
For any $f \in \mathcal{C}_{2,N}$, suppose $f(k)= [f_k:1] \in \mathbb{RP}^1$, let
\begin{equation}
B_k = [f_{k-1}, f_{k+2}, f_{k+1}, f_{k}].
\end{equation}
We have $\{B_k\}_{k=1}^N$ is a coordinate system of $\mathcal{C}_{2,N}$.
\end{definition}
\begin{proposition}
\label{propc2}
{\sc[Formulas for $\{,\}_{\mathcal{C}_{2,N}}$]}
We have
\begin{equation}
\{B_k, B_{k+1}\}_{\mathcal{C}_{2,N}} =  \left(1- \frac{1}{B_k} - \frac{1}{B_{k+1}}\right) B_k \cdot B_{k+1},
\end{equation}
\begin{equation}
\{B_k, B_{k-1}\}_{\mathcal{C}_{2,N}} = - \left(1- \frac{1}{B_k} - \frac{1}{B_{k-1}}\right) B_k \cdot B_{k-1},
\end{equation}
\begin{equation}
\{B_k, B_{k+2}\}_{\mathcal{C}_{2,N}} =  - \frac{1}{B_{k+1}} \cdot B_k \cdot B_{k+2},
\end{equation}
\begin{equation}
\{B_k, B_{k-2}\}_{\mathcal{C}_{2,N}} =   \frac{1}{B_{k-1}} \cdot B_k \cdot B_{k-2},
\end{equation}
for $k =1,...,N$ with the convention $k+N = k$.
 For the other cases
\begin{equation}
\{B_i, B_{j}\}_{\mathcal{C}_{2,N}} = 0.
\end{equation}
\end{proposition}
\begin{proof}
By calculating the swapping bracket, we have the following non-trivial equations
\begin{equation}
\begin{aligned}
&\left[\frac{r_{k-1} r_{k+1}}{r_{k-1} r_{k}}\cdot\frac{r_{k+2} r_{k}}{r_{k+2} r_{k+1}},\frac{r_{k} r_{k+2}}{r_{k} r_{k+1}} \cdot \frac{r_{k+3} r_{k+1}}{r_{k+3} r_{k+2}}\right]
\\&= -1 + \frac{r_{k-1}r_{k+2}\cdot r_{k}r_{k+1}}{r_{k-1}r_{k+1} \cdot r_{k}r_{k+2}} + \frac{r_{k+2}r_{k+1}\cdot r_{k+3}r_{k}}{r_{k+2}r_{k} \cdot r_{k+3}r_{k+1}},
\end{aligned}
\end{equation}
  \begin{equation}
\begin{aligned}
&\left[\frac{r_{k-1}r_{k+1}}{r_{k-1}r_{k}} \cdot\frac{r_{k+2}r_{k}}{r_{k+2}r_{k+1}}, \frac{r_{k+1}r_{k+3}}{r_{k+1}r_{k+2}}\cdot \frac{r_{k+4}r_{k+2}}{r_{k+4}r_{k+3}}\right]
\\&= -\frac{r_{k+2}r_{k+3}\cdot r_{k+1}r_{k}}{r_{k+2}r_{k} \cdot r_{k+1}r_{k+3}}.
\end{aligned}
\end{equation}
Since 
$$\psi\left(\theta\left(\frac{r_{k-1} r_{k+1}}{r_{k-1} r_{k}}\cdot\frac{r_{k+2} r_{k}}{r_{k+2} r_{k+1}}\right)\right)=B_k,$$
$$\psi\left(\theta\left(\frac{r_{k-1}r_{k+2}\cdot r_{k}r_{k+1}}{r_{k-1}r_{k+1} \cdot r_{k}r_{k+2}}\right)\right)=1-\frac{1}{B_{k}},$$
$$\psi\left(\theta\left(\frac{r_{k+2}r_{k+1}\cdot r_{k+3}r_{k}}{r_{k+2}r_{k} \cdot r_{k+3}r_{k+1}}\right)\right)=1-\frac{1}{B_{k+1}},$$
$$\psi\left(\theta\left(\frac{r_{k+2}r_{k+3}\cdot r_{k+1}r_{k}}{r_{k+2}r_{k} \cdot r_{k+1}r_{k+3}}\right)\right)=\frac{1}{B_{k+1}},$$
we conclude that we verify the formulas in the proposition.
\qed
\end{proof}

\begin{remark}
Faddeev-Takhtajan-Volkov algebra coincides with $\{\cdot,\cdot\}_{\mathcal{C}_{2,N}}$, but the original paper \cite{FT86} works on the coordinate system $\{s_i= \frac{1}{B_i}\}_{i=1}^N$ of $\mathcal{C}_{2,N}$. 
\end{remark}
\subsubsection{Case n=3}
\begin{definition}{\sc[coordinate system of $\mathcal{C}_{3,N}$ \cite{SOT10}]}
For any $f \in \mathcal{C}_{3,N}$, let
\begin{equation}\small 
X_k = \left[f(k-2), f(k-2)f(k-1) \wedge f(k+1)f(k+2),f(k-1),f(k-2)f(k-1) \wedge f(k) f(k+1)\right]
\end{equation}
\begin{equation}\small 
Y_k = \left[f(k+1)f(k+2) \wedge f(k-2)f(k-1), f(k+2), f(k+1)f(k+2) \wedge f(k-1)f(k), f(k+1)\right].
\end{equation}
By  R. Schwartz et al.\cite{SOT10}, $\{X_k, Y_k\}_{k=1}^N$ is a coordinate system of $\mathcal{C}_{3,N}$. 
\end{definition}

\begin{proposition}{\sc[Formulas for $\{,\}_{\mathcal{C}_{3,N}}$]}
\label{propc3}
We have
\begin{equation}
\{X_k, X_{k+1}\}_{\mathcal{C}_{3,N}} = X_k \cdot X_{k+1} \cdot (1- Y_k) \cdot (1- X_k - X_{k+1}),
\end{equation}
\begin{equation}
\{X_k, X_{k+2}\}_{\mathcal{C}_{3,N}} = X_k \cdot X_{k+1} \cdot X_{k+2} \cdot (Y_k + Y_{k+1} - 1),
\end{equation}
\begin{equation}
\{X_k, Y_{k-3}\}_{\mathcal{C}_{3,N}} = X_{k-1} \cdot X_k \cdot Y_{k-3} \cdot Y_{k-2},
\end{equation}
\begin{equation}
\{X_k, Y_{k-2}\}_{\mathcal{C}_{3,N}} = X_k \cdot Y_{k-2} \cdot (-X_{k-1} - Y_{k-1}+ X_k Y_{k-1} + X_{k-1} Y_{k-1} + X_{k-1} Y_{k-2}),
\end{equation}
\begin{equation}
\{X_k, Y_{k-1}\}_{\mathcal{C}_{3,N}} = X_k \cdot Y_{k-1} \cdot (1- X_k) \cdot (1- Y_{k-1}),
\end{equation}
\begin{equation}
\{X_k, Y_k\}_{\mathcal{C}_{3,N}} = - X_k \cdot Y_k \cdot (1- X_k) \cdot (1- Y_k),
\end{equation}
\begin{equation}
\{X_k, Y_{k+1}\}_{\mathcal{C}_{3,N}} = X_k \cdot Y_{k+1} \cdot (Y_k + X_{k+1} - X_k Y_k - X_{k+1} Y_k - X_{k+1} Y_{k+1}),
\end{equation}
\begin{equation}
\{X_k, Y_{k+2}\}_{\mathcal{C}_{3,N}} = - X_k \cdot X_{k+1} \cdot Y_{k+1} \cdot Y_{k+2},
\end{equation}
\begin{equation}
\{Y_k, Y_{k+1}\}_{\mathcal{C}_{3,N}} =  Y_k \cdot Y_{k+1} \cdot(1-X_{k+1}) \cdot (1 - Y_k - Y_{k+1}),
\end{equation}
\begin{equation}
\{Y_k, Y_{k+2}\}_{\mathcal{C}_{3,N}} = Y_k \cdot Y_{k+1} \cdot Y_{k+2}(X_{k+1} + X_{k+2} - 1),
\end{equation}
for $k =1,...,N$ with the convention $k+N = k$. Except the above brackets and their symmetry ones by $\{a,b\}_{\mathcal{C}_{3,N}} = -\{b,a\}_{\mathcal{C}_{3,N}}$, all the other brackets between two generators are zero.
\end{proposition}
\begin{proof}
By observing that 
\begin{equation}
\psi\left(\theta\left(\frac{(k-2)(k)}{(k-2)(k+1)}\cdot \frac{(k-1)(k+1)}{(k-1)(k)}\right)\right)=\frac{1}{1-X_k},
\end{equation}
\begin{equation}
\psi\left(\theta\left(\frac{(k+2)(k-1)}{(k+2)(k-2)}\cdot \frac{(k+1)(k-2)}{(k+1)(k-1)}\right)\right)=\frac{1}{1-Y_k}.
\end{equation}
We calculate explicitly the swapping brackets between the $2N$ terms above. By more complicated calculations and the same argument as in Proposition \ref{propc2}, we have the above formulas.
\qed
\end{proof}
\begin{remark}
We notice that the above formulas coincide with Equations (2.2.5) of \cite{P94} when $\tau_{2k-1}$ is replaced by $X_{k}$ and $\tau_{2k}$ is replaced by $Y_k$, where $\{\tau_{k}\}$ are the coordinates for the lattice $W_3$ algebra arising from the study of the quantum invariant ring.
\end{remark}

\section{Large N asymptotic relations}
\label{sectionswapvirn}
There are two Poisson algebras on $\mathcal{C}_{2,N}$, one is the Faddeev-Takhtajan-Volkov algebra, another is induced from the discrete version of the first Adler-Gelfand-Dickey Poissson bracket by Miura transformation.
In this section, we relate asymptotically the dual of these Poisson algebras to the Virasoro algebra.

\subsection{The discrete Hill's operator and the cross-ratios}
\begin{definition}\label{defndh2}
{\sc[Discrete Hill's equation]}
Let $N\geq 1$ be an integer.
Given a periodic sequence $\{H_k\}_{k=-\infty}^\infty$ in $\mathbb{R}$ where $H_{N+k} = H_k$ for any $k \in \mathbb{Z}$. The {\em discrete Hill's equation} is the difference equation in $\{C_k\}_{k=-\infty}^\infty$:
\begin{equation}
\frac{\frac{C_{k+1}-C_{k}}{N}-\frac{C_{k}-C_{k-1}}{N}}{N}  = H_k\cdot C_k,
\end{equation}
or equivalently
\begin{equation}
\label{equHill}
C_{k+1} = \left(\frac{H_k}{N^2} +2 \right) \cdot C_k - C_{k-1},
\end{equation}
for any $k$ belongs to $\mathbb{Z}$.

We define $\{H_k\}_{k = -\infty}^\infty$ to be a {\em discrete Hill's operator}, and $\{C_k\}_{k=-\infty}^\infty$ is the {\em solution} to the discrete Hill's equation.
\end{definition}
Given a discrete Hill's operator, by Equation \ref{equHill}, the series $\{C_k\}_{k=-\infty}^\infty$ is fixed when two initial values $C_0$, $C_1$ are given. Since the difference equation is homogeneous, so up to scalar and $\operatorname{PSL}(2,\mathbb{R})$ projective transformation, there are exactly two linear independent solutions of the discrete Hill's operator, which are corresponding to one point of $\mathcal{C}_{2,N}$.

Conversely, when $N$ is odd, given a point of $\mathcal{C}_{2,N}$, there is one unique Hill's operator corresponding it by the similar argument of Proposition 4.1 of \cite{SOT10}.

\begin{proposition}{\sc[\cite{SOT10}]}
\label{corhb}
Let $N>3$ be odd, let $\{f(i)\}_{i \in \mathbb{Z}} \in \mathcal{C}_{2,N}$, $f(i)= [f_i:1] \in \mathbb{RP}^1$. There exists one unique discrete Hill's equation \ref{equHill} such that $\{X_i\}_{i=-\infty}^\infty$ and $\{Y_i\}_{i=-\infty}^\infty$ are two linear independent solutions of \ref{equHill} and $[X_i:Y_i]= [f_i:1]$.
\end{proposition}
\begin{notation}
Let $b_k= \frac{H_k}{N^2} +2 $.
\end{notation}
\begin{corollary}
Let $N>3$ be odd, $\{b_k\}_{k=1}^N$ and $\{H_k\}_{k=1}^N$ are two coordinate systems of $\mathcal{C}_{2,N}$.
\end{corollary}

The following proposition explains the relation between the discrete Hill's operator and cross ratio coordinate $\{B_k\}_{k=1}^N$.
\begin{proposition}
\label{calH}
Let $\{X_i\}_{i=-\infty}^\infty$ and $\{Y_i\}_{i=-\infty}^\infty$ be two linear independent solutions of \ref{equHill} and $[X_i:Y_i]= [f_i:1]$. For any $ k \in \mathbb{Z}$, we have
\begin{equation}
 B_k = [f_{k-1}, f_{k+2}, f_{k+1}, f_{k}] = b_k \cdot b_{k+1}.
\end{equation}
\end{proposition}
\begin{proof}
Since $\{X_i\}_{i=-\infty}^\infty$ and $\{Y_i\}_{i=-\infty}^\infty$ are linear independent,
we have $$h:= X_1 Y_0 - X_0 Y_1\neq 0.$$
  For any $k \geq 0$, we have
  \begin{equation}
  \begin{aligned}
  &X_{k+1} Y_k -  X_k Y_{k+1} = (b_k X_k -X_{k-1})\cdot Y_k - X_k \cdot(b_k Y_k -Y_{k-1}) = X_{k} Y_{k-1} -  X_{k-1} Y_k
  \\& = ...= X_1 Y_0 - X_0 Y_1 = h,
  \end{aligned}
  \end{equation}
  \begin{equation}
  f_{k+1} - f_k = \frac{X_{k+1}}{Y_{k+1}} -  \frac{X_{k}}{Y_{k}} = \frac{h}{Y_k Y_{k+1}},
  \end{equation}
  \begin{equation}
  \begin{aligned}
  &f_{k+2} - f_k = \frac{X_{k+2}}{Y_{k+2}} -  \frac{X_{k}}{Y_{k}} = \frac{X_{k+2} Y_k - X_k Y_{k+2}}{Y_k Y_{k+2}}
  \\& = \frac{(b_{k+1} X_{k+1}-X_k)\cdot Y_k - X_k \cdot(b_{k+1} Y_{k+1}-Y_k)}{Y_k Y_{k+2}} = \frac{h b_{k+1}}{Y_k Y_{k+1}}.
  \end{aligned}
  \end{equation}
  Thus we have
  \begin{equation}
  \begin{aligned}
  \frac{f_{k-1}-f_{k+1}}{f_{k-1}-f_{k}} \cdot \frac{f_{k+2}-f_{k}}{f_{k+2}-f_{k+1}} = \frac{-\frac{h b_k}{Y_{k-1} Y_{k+1}}}{-\frac{h}{Y_{k-1} Y_{k}}} \cdot \frac{\frac{h b_{k+1}}{Y_{k+2} Y_{k}}}{\frac{h}{Y_{k+2} Y_{k+1}}} = b_k b_{k+1}
  \end{aligned}
  \end{equation}
  for any $k\geq 1$.

  By the similar argument, for $k <1$, we have $B_k=b_k\cdot b_{k+1}$.

  We conclude that $B_k=b_k\cdot b_{k+1}$ for any $k \in \mathbb{Z}$.
  \qed
\end{proof}
\begin{remark}
The cross ratio $B_k$ can be understood as Schwarzian derivative \cite{OT05}.
\end{remark}

\subsection{Virasoro algebra and $\{\cdot,\cdot\}_{\mathcal{C}_{2,N}}$}
By comparing with Proposition 2.3 of \cite{KW09} page 67, we define a discrte version of Virasoro algebra.
\begin{definition}
{\sc[$(t_1,t_2,N)$-Virasoro bracket]} Let $t_1, t_2 \in \mathbb{R}$ and $N \in \mathbb{N}$, the \; {\em $(t_1,t_2,N)$-Virasoro bracket} on $\{I_k\}_{k=-N}^N$ is defined to be:

For $p ,q = -\left[\frac{N-1}{2}\right],...,\left[\frac{N}{2}\right]$,
\begin{enumerate}
  \item when $p \neq -q$, we have
$$\{I_p, I_q\}_{N,t_1,t_2}  = \left(p-q\right)\cdot I_{p+q}  $$
with the convention $I_{k+N} = I_k$;
  \item when $p = -q$, we have
$$\{I_p, I_{-p}\}_{N,t_1,t_2}  = 2p \cdot I_0+ t_1 \cdot p^3 + t_2 \cdot p .$$
\end{enumerate}
\end{definition}

\begin{remark}
Notice that $(t_1,t_2,N)$-Virasoro bracket is asymptotic to the Poisson bracket associated to the 2-cocycle with $c_1 = t_1$, $c_2 = t_2$ as in Proposition 2.3 of \cite{KW09} page 67 when $N$ converges to infinite, but it is not a Poisson bracket.

Very specifical values of $t_1$ and $t_2$ correspond to Virasoro algebra. When $t_1$ is fixed, $t_2$ varies, they correspond to same element in the cohomology group $H^2(\operatorname{Vect}(S^1),\mathbb{R})$. Different $t_1$ corresponds to different element in the one dimensional space $H^2(\operatorname{Vect}(S^1),\mathbb{R})$.
\end{remark}

\begin{definition}{\sc[Discrete Fourier transformation]}
Let $\{B_k\}_{k=1}^N$ be the cross ratio coordinates of $\mathcal{C}_{2,N}$. Let $\mathbb{B}=\{B_1,...,B_N\}$.
The {\em discrete Fourier transformation} $\mathcal{F}$ of $\mathbb{B}$ is defined to be
\begin{equation}
\mathcal{F}_p \mathbb{B}  = \sum_{k=1}^N B_k e^{-\frac{2 p k \pi i}{N}}.
\end{equation}
\end{definition}
Our main proposition of this subsection is
\begin{proposition}{\sc[Large N asymptotic]}
\label{propswapvir}
Let $N>3$.
For $k =-\left[\frac{N-1}{2}\right],...,\left[\frac{N}{2}\right]$, let $$V_k = \frac{\mathcal{F}_{k} \mathbb{B} \cdot N}{8 \pi i}.$$
We have
\begin{equation}
\{V_p, V_q\}_{\mathcal{C}_{2,N}} = \{V_p, V_q\}_{N,\frac{8\pi^2}{N},8N}  + o\left(\frac{1}{N^2}\right).
\end{equation}
\end{proposition}

\begin{proof}

For $p ,q = -\left[\frac{N-1}{2}\right],...,\left[\frac{N}{2}\right]$, we have
\begin{equation}
\begin{aligned}
&\{\mathcal{F}_p \mathbb{B}, \mathcal{F}_q \mathbb{B}\}_{\mathcal{C}_{2,N}} 
\\&= \sum_{k=1}^N  \left(e^{\frac{-2 p k \pi i}{N}} \cdot e^{\frac{-2 q (k+1) \pi i}{N}} - e^{\frac{-2 p (k+1) \pi i }{N}} \cdot e^{\frac{-2 q k \pi i}{N}}\right) \cdot   \left(B_{k} B_{k+1} - B_{k}- B_{k+1}\right)  \\& - \left(e^{\frac{-2 p k \pi i}{N}} \cdot e^{\frac{-2 q (k+2) \pi i}{N}} - e^{\frac{-2 p (k+2) \pi i}{N}} \cdot e^{\frac{-2 q k \pi i}{N}}\right)\cdot \frac{B_{k} B_{k+2}}{B_{k+1}}
\end{aligned}
\end{equation}
By $$B_k = b_k b_{k+1}= 4 + \frac{2 \left(H_k+ H_{k+1}\right) }{N^2}+ \frac{H_k H_{k+1}}{N^4}, $$
 we have
 $$B_k =  4 + \frac{4 H_k }{N^2}+ o\left(\frac{1}{N^2}\right).$$
 We have the above formula equals to
\begin{equation}\small
\begin{aligned}
&  \sum_{k=1}^N   e^{\frac{-2 (p+q) k \pi i}{N}} \cdot \left[\left(\left(1 + \frac{-2 q  \pi i}{N} + \frac{- 2\pi^2 q^2}{N^2} + \frac{ 4\pi^3 q^3 i}{3 N^3}+  o\left(\frac{1}{N^3}\right)\right) -\left(1 + \frac{-2 p  \pi i}{N} + \frac{- 2 \pi^2 \cdot p^2}{N^2} +\right.\right.\right.  \\& \left.\left.\frac{ 4\pi^3 p^3 i}{3 N^3} + o\left(\frac{1}{N^3}\right)\right) \right)\cdot \left(8 + \frac{24 H_k}{N^2}+ o\left(\frac{1}{N^2}\right)\right) - \left(\left(1 + \frac{-4 q  \pi i}{N} + \frac{- 8 \pi^2 q^2}{N^2} +\frac{ 32\pi^3 q^3 i}{3 N^3} + o\left(\frac{1}{N^3}\right)\right)\right.
\\& \left.\left.-\left(1 + \frac{-4 p  \pi i}{N} + \frac{- 8\pi^2 p^2}{N^2} + \frac{ 32\pi^3 p^3 i}{3 N^3} + o\left(\frac{1}{N^3}\right)\right)  \right) \cdot \left(4 + \frac{4 H_k }{N^2} + o\left(\frac{1}{N^2}\right)\right)\right]
\\& = \sum_{k=1}^N   e^{\frac{-2 (p+q) k \pi i}{N}} \cdot \left[\frac{32 H_k (p-q)\pi i}{N^3} - \frac{16 \pi^2 (p^2-q^2)}{N^2} + \frac{ 32\pi^3 (p^3-q^3) i}{ N^3} + o\left(\frac{1}{N^3}\right) \right].
\end{aligned}
\end{equation}
When $p \neq -q$, since $\sum_{k=1}^N   e^{\frac{-2 (p+q) k \pi i}{N}} = 0$, by $B_k =  4 + \frac{4 H_k }{N^2}+ o(\frac{1}{N^2})$, the above formula equals to
\begin{equation}
\sum_{k=1}^N   e^{\frac{-2 (p+q) k \pi i}{N}} \cdot \frac{8 B_k (p-q) \pi i }{N} = \frac{8 (p-q) \pi i}{N}  \cdot \mathcal{F}_{p+q} \mathbb{B} + o\left(\frac{1}{N^3}\right);
\end{equation}
When $p = -q$, the above formula equals to
\begin{equation}
 \begin{aligned}
 &\frac{64 p \pi i }{N^3}\sum_{k=1}^N H_k + \frac{64 p^3 \pi^3 i }{N^2} + o\left(\frac{1}{N^3}\right)
 \\&= \frac{16 p \pi i }{N}\sum_{k=1}^N (B_k-4) + \frac{64 p^3 \pi^3 i }{N^2} + o\left(\frac{1}{N^3}\right)
 \\&= \frac{16 p \pi i}{N} \mathcal{F}_{0} \mathbb{B}  -64 p \pi i  + \frac{64 p^3 \pi^3 i }{N^2} + o\left(\frac{1}{N^3}\right).
 \end{aligned}
\end{equation}
Replacing $\mathcal{F}_{k}$ by $$V_k = \frac{\mathcal{F}_{k} \mathbb{B} \cdot N}{8 \pi i},$$
we obtain that:

for $p \neq -q$,
$$\{V_p, V_q\}_{\mathcal{C}_{2,N}}  = (p-q)\cdot V_{p+q} + o\left(\frac{1}{N^2}\right);$$
for $p = -q$
$$\{V_p, V_{-p}\}_{\mathcal{C}_{2,N}}  =  2p\cdot V_0 + \left( \frac{8 \pi^2  }{N}\right) \cdot p^3 -8N \cdot p + o\left(\frac{1}{N^2}\right).$$
We conclude that
\begin{equation}
\{V_p, V_q\}_{\mathcal{C}_{2,N}} = \{V_p, V_q\}_{N,\frac{8\pi^2}{N},8N} + o\left(\frac{1}{N^2}\right).
\end{equation}
\qed
\end{proof}

\subsection{Poisson structure via Miura transformation and its large N asymptotic}
The discrete version of the Miura transformation $\mu$ \cite{V88} \cite{FRS98}:
\begin{equation}
\mu(v_k) = B_k= \frac{1}{s_k}=(1+v_k)(1+v_{k+1}^{-1})
\end{equation}
is a Poisson map with respect to the Poisson bracket:
\begin{equation}
\{v_i, v_j\}_{V2} = (\delta_{i+1,j}-\delta_{i-1,j})\cdot v_i\cdot v_j
\end{equation}
corresponding to the first Adler-Gelfand-Dickey Poissson bracket.
\begin{definition}{\sc[Another Poisson bracket for $\mathcal{C}_{2,N}$]}
Let us consider the Miura transformation as a map $\mu$ from $\mathbb{R}(v_1,...,v_N)$ to $\mathbb{R}(B_1,...,B_N)$. The map $\mu$ induce another Poisson bracket for $\mathcal{C}_{2,N}$ by:
$$\{B_i , B_j \}_{S2} = \{\mu(v_i), \mu(v_j)\}_{S2}:= \mu \{v_i, v_j\}_{V2}.$$
Thus we have
\begin{equation}
\{B_i, B_j\}_{S2} = (\delta_{i+1,j}-\delta_{i-1,j})\cdot B_i\cdot B_j.
\end{equation}
\end{definition}

We have similar result as Proposition \ref{propswapvir} for $\{\cdot , \cdot \}_{S2}$.
\begin{proposition}{\sc[Large N asymptotic]}
\label{propschvir}
Let $N>3$.
For $k =-\left[\frac{N-1}{2}\right],...,\left[\frac{N}{2}\right]$, let $$W_k = \frac{\mathcal{F}_{k} \mathbb{B} \cdot N}{16 \pi i}.$$
We have
\begin{equation}
\{W_p, W_q\}_{S2} = \{W_p, W_q\}_{N,\frac{8 \pi^2  }{3N},4-8N} + o\left(\frac{1}{N^2}\right).
\end{equation}
\end{proposition}

\begin{proof}

For $p ,q = -\left[\frac{N-1}{2}\right],...,\left[\frac{N}{2}\right]$, we have
\begin{equation}
\begin{aligned}
&\{\mathcal{F}_p \mathbb{B}, \mathcal{F}_q \mathbb{B}\}_{S2} = \sum_{k=1}^N  \left(e^{\frac{-2 p k \pi i}{N}} \cdot e^{\frac{-2 q (k+1) \pi i}{N}} - e^{\frac{-2 p (k+1) \pi i }{N}} \cdot e^{\frac{-2 q k \pi i}{N}}\right) \cdot   \left(B_{k} B_{k+1} \right)
\end{aligned}
\end{equation}
By $$B_k = a_k a_{k+1}= 4 + \frac{2 \left(H_k+ H_{k+1}\right) }{N^2}+ \frac{H_k H_{k+1}}{N^4}, $$
 we have
 $$B_k =  4 + \frac{4 H_k }{N^2}+ o\left(\frac{1}{N^2}\right).$$
 We have the above formula equals to

\begin{equation}
\begin{aligned}
&  \sum_{k=1}^N   e^{\frac{-2 (p+q) k \pi i}{N}} \cdot \left(\left(1 + \frac{-2 q  \pi i}{N} + \frac{- 2\pi^2 q^2}{N^2} + \frac{ 4\pi^3 q^3 i}{3 N^3}+  o\left(\frac{1}{N^3}\right)\right) -\left(1 + \frac{-2 p  \pi i}{N} +\right.\right.  \\& \left.\left. + \frac{- 2 \pi^2 \cdot p^2}{N^2} +\frac{ 4\pi^3 p^3 i}{3 N^3} + o\left(\frac{1}{N^3}\right)\right) \right)\cdot \left(16 + \frac{32 H_k}{N^2}+ o\left(\frac{1}{N^2}\right)\right)
\\&= \sum_{k=1}^N   e^{\frac{-2 (p+q) k \pi i}{N}} \cdot \left[ \frac{32 \pi i (p-q)}{N} + \frac{32 \pi^2 (p^2-q^2)}{N^2}- \frac{ 64\pi^3 (p^3-q^3) i}{ 3 N^3}+ \right.
\\& \left. + \frac{64 H_k (p-q)\pi i}{N^3}  + o\left(\frac{1}{N^3}\right) \right].
\end{aligned}
\end{equation}
When $p \neq -q$, since $\sum_{k=1}^N   e^{\frac{-2 (p+q) k \pi i}{N}} = 0$, by $B_k =  4 + \frac{4 H_k }{N^2}+ o\left(\frac{1}{N^2}\right)$, the above formula equals to
\begin{equation}
\sum_{k=1}^N   e^{\frac{-2 (p+q) k \pi i}{N}} \cdot \frac{16 B_k (p-q) \pi i }{N} = \frac{16 (p-q) \pi i}{N}  \cdot \mathcal{F}_{p+q} \mathbb{B} + o\left(\frac{1}{N^3}\right);
\end{equation}
When $p = -q$, since $\sum_{k=1}^N B_k = 0,$ the above formula equals to
\begin{equation}
\begin{aligned}
&\frac{64 p \pi i}{N} + \frac{128 p \pi i }{N^3}\sum_{k=1}^N H_k -\frac{128 p^3 \pi^3 i }{3N^2} + o\left(\frac{1}{N^3}\right)
\\& = \frac{64 p \pi i}{N} + \frac{32 p \pi i }{N}\sum_{k=1}^N (B_k-4) -\frac{128 p^3 \pi^3 i }{3N^2} + o\left(\frac{1}{N^3}\right)
\\& = \frac{64 p \pi i}{N} + \frac{32 p \pi i }{N}\mathcal{F}_0 \mathbb{B} - 128 p \pi i  -\frac{128 p^3 \pi^3 i }{3N^2} + o\left(\frac{1}{N^3}\right).
\end{aligned}
\end{equation}
Thus we have:
\begin{enumerate}
  \item for $p \neq -q$
$$\{W_p, W_q\}_{S2}  = \left(p-q\right)\cdot W_{p+q} + o\left(\frac{1}{N^2}\right) ;$$
  \item for $p = -q$
$$\{W_p, W_{-p}\}_{S2}  =  2p\cdot W_0 + (4-8N)\cdot p -  \frac{8 \pi^2  }{3N}\cdot  p^3  + o\left(\frac{1}{N^2}\right).$$
\end{enumerate}
We conclude that
\begin{equation}
\{W_p, W_q\}_{S2} = \{W_p, W_q\}_{N,\frac{8 \pi^2  }{3N},4-8N}  + o\left(\frac{1}{N^2}\right).
\end{equation}
\qed
\end{proof}
\section{Two Poisson structures are compatible}
With the evidences shown in Proposition \ref{propswapvir} \ref{propschvir}, we prove that $\{\cdot, \cdot\}_{\mathcal{C}_{2,N}}$ and $\{\cdot, \cdot\}_{S2}$ are compatible in this section.
Firstly, let us recall the traditional definition of the bihamiltonian system.
\begin{definition}{\sc[\cite{KW09}, p47, definition 4.19]}
 Two Poisson brackets $\{\cdot, \cdot\}_a$ and $\{\cdot, \cdot\}_b$ for a manifold $M$ are said to be compatible if and only if  for any $\lambda$, $\{\cdot, \cdot\}_a + \lambda \{\cdot, \cdot\}_b$ is Poisson.\\
 A dynamic system $\frac{d}{dt} m = \xi(m)$ over $M$ is bihamiltonian if its vector field $\xi$ is Hamiltonian with respect to these two Poisson brackets $\{\cdot, \cdot\}_a$ and $\{\cdot, \cdot\}_b$.
\end{definition}
Following the above definition, we make small modification for a ring $R \subset C^\infty(M, \mathbb{R})$.
\begin{definition}
 Two Poisson brackets $\{\cdot, \cdot\}_a$ and $\{\cdot, \cdot\}_b$ on a ring $R \subset C^\infty(M, \mathbb{R})$ are said to be compatible if and only if  for any $\lambda \in \mathbb{R}$, $\{\cdot, \cdot\}_a + \lambda \{\cdot, \cdot\}_b$ is Poisson.\\
\end{definition}
By definition, it is easy to verify that
\begin{proposition}
$\{ \cdot , \cdot \}_a$ and $\{ \cdot , \cdot \}_b$ are compatible if and only if for any $ x,y,z \in R$, we have
\begin{equation}
\begin{aligned}
&\{\{ x, y\}_a, z \}_b + \{\{ y, z\}_a, x \}_b+ \{\{ z, x\}_a, y \}_b + \nonumber \\ & +\{\{ x, y\}_b, z \}_a + \{\{ y, z\}_b, x \}_a + \{\{ z, x\}_b, y \}_a = 0
\end{aligned}
\end{equation}
\end{proposition}
Our main result of this paper is the following.
\begin{theorem}{\sc[Main result]}
\label{thmcomswapsch}
For $N \geq 5$,
$\{\cdot, \cdot\}_{\mathcal{C}_{2,N}}$ and $\{\cdot, \cdot\}_{S2}$ are compatible on $\mathbb{R}(B_1,...,B_N)$.
\end{theorem}
\begin{proof}

Let
\begin{equation}
\begin{aligned}
&K(B_i,B_j,B_k) : = \{\{ B_i, B_j\}_{\mathcal{C}_{2,N}}, B_k \}_{S2} + \{\{ B_j, B_k\}_{\mathcal{C}_{2,N}}, B_i \}_{S2}+ \{\{ B_k, B_i\}_{\mathcal{C}_{2,N}}, B_j \}_{S2}
\\&+ \{\{ B_i, B_j\}_{S2}, B_k \}_{\mathcal{C}_{2,N}} + \{\{ B_j, B_k\}_{S2}, B_i \}_{\mathcal{C}_{2,N}}+ \{\{ B_k, B_i\}_{S2}, B_j \}_{\mathcal{C}_{2,N}}.
\end{aligned}
\end{equation}
By definition, we have to check that
$$
K(B_i,B_j,B_k)  = 0
$$
for any $i,j,k =1,...,N$.
Since $$K(B_i,B_j,B_k) = -K(B_j,B_i,B_k),$$
when some indexes coincide, for example $i = j$, we have
$$
K(B_i,B_i,B_k)  = 0.
$$
Let $\sigma_s$ be the permutation of the $N$ indexes such that $\sigma(l) = l+s$. The permutation $\sigma_s$ induce a ring automorphism $\chi_s$ of $\mathbb{R}(B_1,...,B_N)$ such that $$\chi_s(B_l) = B_{l+s}$$ for $l=1,...,N$.
Moreover, we have
$$\{\chi_s(B_i),\chi_s(B_j)\}_{S2} = \chi_s(\{B_i,\;B_j \}_{S2})$$
and
$$\{\chi_s(B_i),\chi_s(B_j)\}_{\mathcal{C}_{2,N}} = \chi_s(\{B_i,\;B_j \}_{\mathcal{C}_{2,N}}).$$
 Let $\tau$ be the permutation of the $N$ indexes such that $$\tau(l) = N+1-l$$ for $l=1,...,N$. The permutation $\tau$ induce a ring automorphism $\nu$ of $\mathbb{R}(B_1,...,B_N)$ such that $$\nu(B_l) = B_{N+1-l}.$$
Moreover, we have
$$\{\nu(B_i),\nu(B_j)\}_{S2} = -\nu(\{B_i,\;B_j \}_{S2}),$$
$$\{\nu(B_i),\nu(B_j)\}_{\mathcal{C}_{2,N}} = -\nu(\{B_i,\;B_j \}_{\mathcal{C}_{2,N}}).$$
 By the above symmetry, we suppose that $$i=1$$ and $$1  < j < k \leq N .$$ Let $$l := \min\{|j-i|,\; |j-i-N|,\; |k-j|,\; |k-j-N|,\; |i-k|,\; |i-k-N|\}.$$ We suppose that $l = |j-1|$, we have to verify the following cases:
\begin{enumerate}
\item When $1<j-1< k-2< N-1$, we have
\begin{equation}
\begin{aligned}
&K(B_1,B_i,B_k)
\\&=\{\{ B_1, B_j\}_{\mathcal{C}_{2,N}}, B_k \}_{S2} + \{\{ B_j, B_k\}_{\mathcal{C}_{2,N}}, B_1 \}_{S2}+ \{\{ B_k, B_1\}_{\mathcal{C}_{2,N}}, B_j \}_{S2}
\\&+ \{\{ B_1, B_j\}_{S2}, B_k \}_{\mathcal{C}_{2,N}} + \{\{ B_j, B_k\}_{S2}, B_1 \}_{\mathcal{C}_{2,N}}+ \{\{ B_k, B_1\}_{S2}, B_j \}_{\mathcal{C}_{2,N}}
\\&= \{\{ B_1, B_j\}_{\mathcal{C}_{2,N}}, B_k \}_{S2} + \{\{ B_j, B_k\}_{\mathcal{C}_{2,N}}, B_1 \}_{S2}+ \{\{ B_k, B_1\}_{\mathcal{C}_{2,N}}, B_j \}_{S2}.
\end{aligned}
\end{equation}
Since $$\{ B_1, B_j\}_{\mathcal{C}_{2,N}}$$ is a polynomial of $B_1,...,B_j$, we have $$\{\{ B_1, B_j\}_{\mathcal{C}_{2,N}}, B_k \}_{S2} = 0.$$ Similarly, we have $$\{\{ B_j, B_k\}_{\mathcal{C}_{2,N}}, B_1 \}_{S2} = 0$$ and  $$\{\{ B_k, B_1\}_{\mathcal{C}_{2,N}}, B_j \}_{S2}=0.$$
We conclude that $$K(B_1,B_i,B_k) = 0.$$
  \item
  When $j = 2$, $k = 3$,
  we have
\begin{equation}
\begin{aligned}
&K(B_1,B_2,B_3)
\\&=\{\{ B_1, B_2\}_{\mathcal{C}_{2,N}}, B_3 \}_{S2} + \{\{ B_2, B_3\}_{\mathcal{C}_{2,N}}, B_1 \}_{S2}+ \{\{ B_3, B_1\}_{\mathcal{C}_{2,N}}, B_2 \}_{S2}
\\&+ \{\{ B_1, B_2\}_{S2}, B_3 \}_{\mathcal{C}_{2,N}} + \{\{ B_2, B_3\}_{S2}, B_1 \}_{\mathcal{C}_{2,N}}+ \{\{ B_3, B_1\}_{S2}, B_2 \}_{\mathcal{C}_{2,N}}
\\&= \{B_1 B_2 - B_1 - B_2 , B_3 \}_{S2} + \{B_2 B_3 - B_2 - B_3 , B_1 \}_{S2}+ \{\frac{B_3 B_1}{B_2}, B_2 \}_{S2}
\\&+ \{  B_1  B_2 , B_3 \}_{\mathcal{C}_{2,N}} + \{ B_2  B_3 , B_1 \}_{\mathcal{C}_{2,N}}
\\&= (B_1-1)B_2 B_3 - (B_3-1)B_1 B_2 + B_1 (B_2 B_3 -B_2- B_3) \\&- \frac{B_1 B_3}{B_2} B_2- (B_1 B_2 - B_1 - B_2)B_3 + \frac{B_3 B_1}{B_2} B_2
\\&=0.
\end{aligned}
\end{equation}
  \item
When $j =2$, $k = 4$ and $N>5$,  we have
\begin{equation}
\begin{aligned}
&K(B_1,B_2,B_4)
\\&=\{\{ B_1, B_2\}_{\mathcal{C}_{2,N}}, B_4 \}_{S2} + \{\{ B_2, B_4\}_{\mathcal{C}_{2,N}}, B_1 \}_{S2}+ \{\{ B_4, B_1\}_{\mathcal{C}_{2,N}}, B_2 \}_{S2}
\\&+ \{\{ B_1, B_2\}_{S2}, B_4 \}_{\mathcal{C}_{2,N}} + \{\{ B_2, B_4\}_{S2}, B_1 \}_{\mathcal{C}_{2,N}}+ \{\{ B_4, B_1\}_{S2}, B_2 \}_{\mathcal{C}_{2,N}}
\\&= \{B_1 B_2 - B_1 - B_2 , B_4 \}_{S2} + \{-\frac{B_2 B_4}{B_3}, B_1 \}_{S2} + \{  B_1  B_2 , B_4 \}_{\mathcal{C}_{2,N}}
\\&= \frac{B_1 B_2 B_4}{B_3} - \frac{B_1 B_2 B_4}{B_3}
\\&=0.
\end{aligned}
\end{equation}
  \item
When $j =2$, $k = 4$ and $N=5$,  we have
\begin{equation}
\begin{aligned}
&K(B_1,B_2,B_4)
\\&=\{\{ B_1, B_2\}_{\mathcal{C}_{2,N}}, B_4 \}_{S2} + \{\{ B_2, B_4\}_{\mathcal{C}_{2,N}}, B_1 \}_{S2}+ \{\{ B_4, B_1\}_{\mathcal{C}_{2,N}}, B_2 \}_{S2}
\\&+ \{\{ B_1, B_2\}_{S2}, B_4 \}_{\mathcal{C}_{2,N}} + \{\{ B_2, B_4\}_{S2}, B_1 \}_{\mathcal{C}_{2,N}}+ \{\{ B_4, B_1\}_{S2}, B_2 \}_{\mathcal{C}_{2,N}}
\\&= \{B_1 B_2 - B_1 - B_2 , B_4 \}_{S2} + \{-\frac{B_2 B_4}{B_3}, B_1 \}_{S2}+ \{-\frac{B_4 B_1}{B_5}, B_2 \}_{S2}
\\&+ \{  B_1  B_2 , B_4 \}_{\mathcal{C}_{2,N}}
\\&= \frac{B_1 B_2 B_4}{B_3} - \frac{B_1 B_4 B_2}{B_5}- \frac{B_1 B_2 B_4}{B_3} + \frac{B_1 B_4 B_2}{B_5}
\\&=0.
\end{aligned}
\end{equation}
\end{enumerate}
We conclude that for $N \geq 5$,
$\{\cdot, \cdot\}_{\mathcal{C}_{2,N}}$ and $\{\cdot, \cdot\}_{S2}$ are compatible on $\mathbb{R}(B_1,...,B_N)$.
\qed
\end{proof}
\begin{acknowledgements}
The author is very grateful to Fran\c cois Labourie for suggesting the subject and for the guidance. He thanks Vladimir Fock for very interesting conversations and useful comments on this work. He is grateful to University of Paris-Sud, Max Planck Institute for Mathematics and Yau Mathematical Sciences Center for their hospitality.
\end{acknowledgements}

\end{document}